\documentclass[12pt]{amsart}

\title{Obstructions for normally spanned sets of vertices}
\author{Nicola Lorenz}
\author{Max Pitz}
\address{Universit\"at Hamburg, Department of Mathematics, Bundesstrasse 55 (Geomatikum), 20146 Hamburg, Germany}
\email{\{nicola.lorenz, max.pitz\}@uni-hamburg.de}

\usepackage{amsmath,amssymb,amsthm}
\usepackage{mathtools}
\usepackage{tikz}
\usepackage{comment}
\usepackage{enumitem}
 \usepackage{url}
\usepackage[noadjust]{cite}

\usepackage{appendix}

\usepackage{xcolor} 
\colorlet{darkishRed}{red!80!black}
\colorlet{darkishBlue}{blue!60!black}
\colorlet{darkishGreen}{green!60!black}
\usepackage{hyperref}
\hypersetup{
	colorlinks,
	linkcolor={red!60!black},
	citecolor={green!60!black},
	urlcolor={blue!60!black},
}
\usepackage[abbrev, msc-links]{amsrefs}
\usepackage{cite}
\usepackage[capitalise, nameinlink]{cleveref}

\usepackage[utf8]{inputenc}
\usepackage[T1]{fontenc}
\usepackage{lmodern}
\usepackage[babel]{microtype}
\usepackage[english]{babel}

\usepackage{geometry}
\let\polishlcross=\l
\def\l{\ifmmode\ell\else\polishlcross\fi}

\let\emptyset=\varnothing

\let\theta=\vartheta
\let\rho=\varrho
\let\phi=\varphi

\def\NN{\mathbb N}

\def\cV{{\mathcal V}}

\newcommand{\set}[2]{{\{ #1 : #2 \}}}
\newcommand{\singleset}[1]{{\{#1\}}}

\newcommand{\cardinality}[1]{{\left\lvert {#1} \right\rvert}}

\def\downcl#1{\lceil{#1}\rceil}
\def\upcl#1{\lfloor{#1}\rfloor}

\renewcommand{\subset}{\subseteq}

\renewcommand{\triangleleft}{\vartriangleleft}
\newcommand{\nottriangleleft}{\not\kern-1pt\mathrel{\triangleleft}}

\theoremstyle{plain}
\newtheorem{thm}{Theorem}[section]

\newtheorem{lemma}[thm]{Lemma}

\theoremstyle{definition}

\DeclareMathOperator{\cf}{cf}
\DeclareMathOperator{\height}{height}

\makeatletter
\newcommand{\itemlabel}[2]{\item[#2]\begingroup\def\@currentlabel{#2}\phantomsection\label{#1}\endgroup}

\begin{document}

\begin{abstract}
Halin conjectured that a graph has a normal spanning tree if and
only if every minor of it has countable colouring number. This has recently been proven by the second author. 

In this paper, we strengthen this result by establishing the following local version of it: Given a prescribed set of vertices $U$ in a connected graph $G$, there is a normal tree in $G$ that includes $U$ if and only if every $U$-rooted minor of $G$ (i.e.\ a minor every branch set of which meets $U$) has countable colouring number. 

Our proof relies on a novel approach that combines normal partition trees as introduced by Brochet and Diestel with a suitable closure argument developed by Robertson, Seymour and Thomas in their discussion of infinite graphs of finite tree width.
\end{abstract}

\keywords{Normal spanning tree, colouring number}


\subjclass[2020]{05C63, 05C69, 05C83}

\vspace*{24pt} 

\maketitle

\section{Introduction}

A rooted spanning tree $T$ of a connected graph $G$ is \textbf{normal} if the endvertices of all edges of $G$ are comparable in the tree order on $T$. In other words, we require that all the edges of $G$ run `parallel' to branches of $T$, but never `across'.
Every finite connected graph has a normal spanning tree (also known as a depth-first search tree) and so do all countable connected graphs. However, not all graphs have normal spanning trees: consider, for example, any uncountable clique.

Normal spanning trees are among the most useful tools for dealing with infinite graphs. They play a crucial role in the Robertson-Seymour-Thomas characterisation \cite{robertson1992excluding} of graphs without subdivided infinite clique in terms of tree-decompositions of finite width; see in particular \cite{diestel1994depth} and the references therein.
This correspondence is used, for example, in Berger's result that graphs without subdivided infinite clique have an unfriendly partition \cite{berger2017unfriendly,Niel_unfriendly}.
For further applications of normal spanning trees, see e.g.\ \cite{diestel1996classification, diestel1992proof, diestel2006end}.

Having observed that the property of having a normal spanning tree is closed under taking (connected) minors, Halin \cite{Halin} raised the question to characterise these forbidden minors. 
Following Erd\H{o}s and Hajnal \cite{erdHos1966chromatic}, a graph has \textbf{countable colouring number} if one can well-order its vertices so that every vertex is preceded by only finitely many of its neighbours. 
In 2021, the second author answered Halin's question as follows:

\begin{thm}[{\cite{Pitz}}]
\label{thm_main1}
   A connected graph has a normal spanning tree if and only if every minor of it has countable colouring number.
\end{thm}

Then the forbidden subgraphs for countable colouring number found by Bowler, Carmesin, Komj\'ath and Reiher (\cite{TheColouringNumberOfInfiniteGraphs}) form the forbidden minors for having a normal spanning tree. 

The concept of a normal spanning tree extends to subtrees which are not necessarily spanning as follows: a rooted tree $T$ in a graph $G$ is \textbf{normal} if for every $T$-path $P$ in $G$ (a path with both endvertices but no inner vertices or edges in $T$), the endvertices of $P$ are comparable in the tree order.
If $T$ is spanning, this reduces to the requirement that the ends of any edge of $G$ are comparable in the tree order on $T$.

The importance of normal trees (that are not necessarily spanning) arises from the fact that virtually all existence proofs for normal spanning trees proceed by building a (potentially transfinite) sequence of larger and larger normal subtrees until all vertices are incorporated and the final normal tree in the sequence is spanning. However, normal trees that cover large, prescribed sets of vertices are also important in their own right, especially in cases where the whole graph does not admit a normal spanning tree. Kurkofka, Melcher and Pitz have shown that all graphs can be approximated by normal trees up to an arbitrarily prescribed error term, an assertion which has a number of implications about the topological structure of infinite graphs \cite{kurkofka2021approximating}. In \cite{burger2022end}, normal trees play an important role in finding end-faithful spanning trees in connected graphs without topological $T_{\aleph_1}$ minors, and in the star-comb series \cite{StarsAndCombs1,StarsAndCombs2,StarsAndCombs3,StarsAndCombs4}, normal trees occur frequently as duals to certain forbidden structures.

At the heart of these applications lies the question for which prescribed sets of vertices $U$ there exists a normal tree $T$ in $G$ with $U \subseteq V(T)$. If the answer is yes, then we say that $U$ is \textbf{normally spanned} in $G$, and the question arises whether one can find a local version of Theorem~\ref{thm_main1} above characterising which sets of vertices are normally spanned.

To answer this question, we need a concept of how minors of $G$ interact with a prescribed set of vertices $U$: A minor of $G$ is said to be \textbf{$U$-rooted} if every branch set contains a vertex of $U$. Rooted minors have been studied, for example, in \cite{RootedMinors}, \cite{RootedMinorProblemsKen} and \cite{RootedMinorsPaulWollan}. If a set of vertices $U$ is normally spanned in a graph $G$, then a routine argument (see Section~\ref{sec_easy} below) shows that every $U$-rooted, connected minor of $G$ admits a normal spanning tree, and thus has countable colouring number (enumerate the vertices of the normal tree level by level).

The main result of this paper is a converse to this observation, establishing the following local version of Theorem~\ref{thm_main1}:

\begin{thm}
\label{thm_main2}
    A set of vertices $U$ of a connected graph $G$ is normally spanned if and only if every $U$-rooted minor of $G$ has countable colouring number.
\end{thm}

The proof generally follows the approach to the proof of Theorem~\ref{thm_main1} from \cite{Halin}, with one crucial change that results in a stronger, yet significantly shorter overall proof. Proceeding by induction on $|G|$ (resp.~$|U|$), one would like to leverage the existence of normal trees of connected subgraphs $H \subset G$ with $|H| < |G|$. However, the problem is that a tree that is normal within some subgraph $H$ is not necessarily normal within all of $G$. Generalising an idea from \cite{QuicklyProvingDiestelsNSTcriterion}, we identify a sufficient condition (what we call below \emph{well-connected adhesion sets}) for a subgraph $H$ so that every normal subtree of $H$ is normal in $G$, see Lemma~\ref{lem_normalinsubgraphs} below. We then use a relaxed version of the \emph{normal partition trees} introduced by Brochet and Diestel \cite{NormalTreeOrdersForInfiniteGraphs}, and combine it with a countable closure argument developed by Robertson, Seymour and Thomas \cite[Assertion (2.2)]{robertson1992excluding} in their discussion of tree decompositions of finite width in order to find a suitable decomposition of $G$ into subgraphs $H$ all having  \emph{well-connected adhesion sets}.

\subsection{Organisation of this paper}

In Section~\ref{sec_prelims}, we recall all facts about normal trees needed in this paper and then prove that our condition in Theorem~\ref{thm_main2} for characterising normally spanned sets of vertices is necessary.

The rest of the paper is then concerned with the converse implication, that our condition is also sufficient.
Section~\ref{sec_wcas} introduces our crucial notion of \emph{well-connected adhesion sets} and describes its relation to normal trees. We then demonstrate how to build subgraphs with well-connected adhesion sets by a countable closure argument. 

In Section~\ref{sec_seminpt} we introduce our mechanism for structuring infinite graphs, introducing the concept of \emph{$T$-graphs} and \emph{normal (semi-) partition trees}. This section concludes with our main existence result for normal partition trees tailored around the prescribed set of vertices $U$ and incorporating the property of having \emph{well-connected adhesion sets} in $G$. 

Section~\ref{sec_declemma} establishes a decomposition lemma, essentially saying that our graph $G$ can be written as a continuous chain of certain, well-behaved subgraphs $H_i$ with $|H_i| < |G|$, facilitating an inductive proof. 
Finally, Section~\ref{sec_finalproof} brings all the previous ingredients together and completes the proof of our forbidden minor characterisation, as stated in Theorem~\ref{thm_main2}.

\section{The forwards implication}
\label{sec_prelims}

In this section, we recall well-known facts about normal trees that allow us to complete the proof of the necessity of our characterisation in Theorem~\ref{thm_main2}.

\subsection{Background on normal trees}

If $T$ is a graph-theoretic tree with root~$r$, we write $x \le y$
for vertices $x,y\in T$ if $x$ lies on the unique $r$--$y$ path in~$T$.
A rooted tree $T \subset  G$ \emph{contains a set of vertices $U$} if $U \subset V(T)$. For a node $t$ of the rooted tree $T$, we write $\lceil t \rceil = \lceil t \rceil_T := \{t' \in T \colon t'\le t\}$ for its \textbf{down-closure}, the set of vertices in $T$ that lie on the unique path from $t$ to the root $r$.

A rooted tree $T \subset G$ is \textbf{normal (in $G$)} if the end vertices of any $T$-path in $G$ (a path in $G$ with end vertices in $T$ but all edges and inner vertices outside of $T$) are comparable in the tree order of $T$.  For a normal tree $T\subset G$, the neighbourhood $N(D)$ of every component $D$ of $G-T$ forms a chain in $T$. For later use, we recall the following lemma, see also \cite[\S1.5 and \S8.2]{Bibel}:

\begin{lemma}
\label{lem_normalsep}
    Let $T$ be a normal tree in a graph $G$. Then for any two incomparable vertices $t,t'$ in $T$, the finite set $\downcl{t} \cap \downcl{t'}$ separates $t$ from $t'$ in $G$. \hfill \qed
\end{lemma}

A set of vertices $U \subset V(G)$ is \textbf{dispersed} (in $G$) if every ray in $G$ can be separated from $U$ by a finite set of vertices. 
The following theorem of Jung \cite[Satz~6]{Jung} characterises the normally spanned sets of vertices. See also \cite{pitz2020unified} for a modern treatment.

\begin{thm}[Jung]
\label{thm_Jung}
    Let $G$ be a connected graph and $r$ any vertex of $G$. Then a set of vertices $U \subset V(G)$ is contained in some normal tree of $G$ with root $r$ if and only if $U$ is a countable union of dispersed sets in $G$. \hfill \qed
\end{thm}

In particular, if a set of vertices $U$ is normally spanned, then any vertex $r$ of $G$ may serve as the root of a normal tree containing $U$.

\subsection{Proof of the first implication}
\label{sec_easy}

We prove the direct implication in Theorem~\ref{thm_main2}. Let $U$ be a set of vertices in a graph $G$.
We claim that in the following sequence of assertions, every one of them implies the next. 

\begin{enumerate}
    \item $U$ is normally spanned in $G$,
    \item every connected $U$-rooted minor of $G$ has a normal spanning tree,
    \item every $U$-rooted minor of $G$ has countable colouring number,
    \item every $U$-rooted minor of $G$  with countable branch sets has countable colouring number.
\end{enumerate}

Indeed, for $(1) \Rightarrow (2)$, let $T$ be the normal tree in $G$ that covers $U$. By Jung's Theorem~\ref{thm_Jung}, we can write $U$ as a countable union of dispersed sets $\{U_n\colon n \in \NN\}$. Let $H$ be a connected $U$-rooted minor of $G$ and for every $n \in \NN$ let
$$
U'_n := \set{v \in V(H)}{X_v \cap U_n \neq \emptyset}.
$$

Then $V(H) = \bigcup_{n \in \NN} U'_n$. Hence, again by Jung's Theorem~\ref{thm_Jung}, we are done once we have shown that each $U'_n$ is dispersed in $H$. Thus, let $R'$ be a ray in $H$ and define the following connected induced subgraph $R^* := \set{ v \in V(G) }{ \exists w \in V(R'): v \in X_w }$ in $G$. Let $R \subseteq R^*$ be a ray in $G$ such that $R$ meets every branch set $X_w$ for $w \in R'$. Let $S \subseteq V(G)$ be the finite set of vertices that separates $U_n$ and $R$ in $G$. Next, define $S' = \set{v \in V(H)}{X_v \cap S \neq \emptyset}$. Since $S$ is finite, $S'$ is a finite set of vertices in $H$. We show that $S'$ separates $U'_n$ and $R'$ in $H$. Assume for a contradiction that there is an $U'_n$--$R'$ path $P'$ that does not meet $S'$. Consider $P^* := \set{v \in V(G)}{\exists w \in V(P'): v \in X_w}$ similarly as before. In $P^*$ we find an $U_n$--$R$ path $P$ that does not meet $S$ by construction. A contradiction.

Since $(3) \Rightarrow (4)$ is trivial, we remains to prove $(2) \Rightarrow (3)$. 
So let $H$ be any $U$-rooted minor of $G$. We need to define a well-order on $V(H)$ witnessing that $H$ has countable colouring number. By considering each component individually, we may assume that $H$ is connected; by (2), $H$ has a normal spanning tree $T$. Let $T^i$ be the $i$th level of $T$ and fix an arbitrary well-order $\leq_i$ of $T^i$ for all $i \in \NN$. Let $v, v' \in V(H)$ and let $i, i' \in \NN$ such that $v \in T^i$ and $v' \in T^{i'}$. We define $v \leq v'$ if and only if $i < i'$ or $i = i'$ and $v \leq_i v'$. This defines a well-order $(V(H), \leq)$ of $V(T) = V(H)$.

Next, consider an arbitrary vertex $v \in V(H)$ and find $i \in \NN$ such that $v \in T^i$. Then all neighbours $w$ of $v$ with $w \leq v$ must be contained in $\bigcup_{j < i} T^j$ by definition of $(V(H), \leq)$, so must belong to $\lceil v \rceil_T$ by the defining property of a normal spanning tree. Since $\lceil v \rceil_T$ is finite, we have shown that there are at most finitely many neighbours $w$ of $v$ with $w \leq v$. As $v$ was arbitrary, this shows that $H$ has countable colouring number.

\section{Well-connected adhesion sets and normal trees}
\label{sec_wcas}

For distinct vertices $v,w$ of $G$ we denote by $\kappa(v,w) = \kappa_G(v,w)$ the \textbf{connectivity between $v$ and $w$} in $G$, i.e.\  the largest size of a family of independent (i.e.\ pairwise internally-disjoint) $v-w$ paths. If $v$ and $w$ are non-adjacent, this is by Menger's theorem for infinite graphs  \cite[Proposition~8.4.1]{Bibel} equivalent to the minimal size of a $v-w$ separator in $G$. 

Suppose $H$ is an induced subgraph of a graph $G$. We call a set of vertices  $A \subseteq V(H)$ an \textbf{adhesion set} of $H$ in $G$ if there is a component $D$ of $G - H$ such that $A = N(D)$.

We say that a subgraph $H$ of $G$ has \textbf{well-connected adhesion sets} if the end vertices of any $H$-path in $G$ have infinite connectivity in $H$. Note that if $H$ is an induced subgraph, then the endvertices of any $H$-path lie in the same adhesion set.

\begin{lemma}
\label{lem_normalinsubgraphs}
    Let $H$ be a subgraph of $G$ with well-connected adhesion sets. Then any rooted subtree of $H$ that is normal in $H$ is also normal in $G$.
\end{lemma}

\begin{proof}
    Let $T$ be a rooted subtree of $H$ that is normal in $H$.
    Suppose for a contradiction that $T$ is not normal in $G$; then the end-vertices $v$ and $w$ of some $T$-path $P$ in $G$ are incomparable in the tree order of $T$. Since $T$ is normal in $H$, we know by Lemma~\ref{lem_normalsep} that $X:= \lceil v \rceil_T \cap \lceil w \rceil_T$ is a finite separator of $v$ and $w$ in $H$. Now let $P_1,\ldots,P_k$ be the collection of $H$-paths included in $P$, where $P_i$ has end-vertices $v_i$ and $w_i$, say. Since $H$ has well-connected adhesion sets, it follows that for each $i$ there exists an $v_i-w_i$ path $S_i \subset H$ that avoids the finite set $X$. But then $P \cap H$ together with all $S_i$ forms a connected subgraph of $H$ that contains $v$ and $w$ but avoids $X$, a contradiction.
\end{proof}

The property of having well-connected adhesion sets can be achieved as follows; for later use, we also incorporate a connectivity requirement.

\begin{lemma}
\label{lem_wellconnected_new}
    Let $C$ be a set of vertices in a graph $G$ such that $G - C$ is connected. Then every set of vertices $W \supseteq C$ is included in a superset $\widehat{W} \supseteq W $ such that  
    
    \begin{enumerate}
        \item $|\hat{W}| \leq |W| + \aleph_0$,
        \item $G[\widehat{W}]$ has well-connected adhesion sets, and 
        \item $G[\widehat{W} \setminus C]$ is connected.
    \end{enumerate}
    
\end{lemma} 

\begin{proof}
    We construct $\widehat{W}$ from $W$ by a countable closure argument. Define a sequence of sets of vertices 
    $$
    W_0 \subset W_1 \subset W_2 \subset \cdots
    $$
    with $|W_n| = |W|$ as follows: 
    Let $W_0 = W$, and supposing that $W_{2n}$ has already been defined, construct $W_{2n + 1}$ by adding for every pair $v, w \in W_{2n}$ with finite connectivity in $G$ (the vertex set of) a maximal family of independent $v - w$ paths in $G$ to $W_{2n}$, and for all remaining pairs some $\aleph_0$ many independent  $v - w$ paths in $G$ to $W_{2n}$. To construct $W_{2n + 2}$, pick an inclusion minimal tree $T_n$ with $W_{2n} \setminus C \subset T_n \subset G - C$, and let $W_{2n + 2} = W_{2n + 1} \cup T_n$. 
    
    Then $\widehat{W} = \bigcup_{n \in \NN} W_n$ is as desired: Clearly, $|\hat{W}| \leq |W| + \aleph_0$ by construction, so (1) holds. Any $\widehat{W}$-path from $v$ to $w$ witnesses that the connectivity of $v$ and $w$ was infinite, and so we have added infinitely many independent $v-w$ paths to $\widehat{W}$ in the process, giving (2). Furthermore, $G[\widehat{W} \setminus C] = \bigcup_{n \in \NN} T_n$ is connected, giving (3).
\end{proof}

\section{Normal semi-partition trees}
\label{sec_seminpt}

\subsection{Normal tree orders and $T$-graphs}

A partially ordered set $(T, \leq)$ is called an \textbf{order tree} if it has a unique minimal element (called the \textbf{root}) and all subsets of the form $\lceil t \rceil = \lceil t \rceil_T := \{t' \in T \colon t'\le t\}$ are well-ordered. Our earlier partial ordering on the vertex set of a rooted graph-theoretic tree is an order tree in this sense, where all $\lceil t \rceil$ are finite.

Let $T$ be an order tree. A~maximal chain in~$T$ is called a \textbf{branch} of~$T$; note that every branch inherits a well-ordering from~$T$. The \textbf{height} of~$T$ is the supremum of the order types of its branches. The \textbf{height} of a point $t\in T$ is the order type of~$\mathring{\lceil t \rceil}  := \lceil t \rceil  \setminus \{t\}$. The set $T^i$ of all points at height $i$ is the $i$th \textbf{level} of~$T$, and we write $T^{<i} := \bigcup\{T^j \colon j < i\}$. An \textbf{Aronszajn tree} is an uncountable tree with no uncountable levels and no uncountable branches.

The intuitive interpretation of a tree order as expressing height will also be used informally. For example, we say that $t$ is \textbf{above}~$t'$ if $t > t'$, and call $\lceil X \rceil = \lceil X \rceil _T := \bigcup \set{\lceil x \rceil}:{x \in X}$ the \textbf{down-closure} of~$X\subseteq T$. And we say that $X$ is \textbf{down-closed}, or $X$ is a \textbf{rooted subtree}, if $X=\lceil X \rceil $. If $t < t'$, we write $[t,t'] = \{x \colon t \leq x \leq t'\}$, and call this set a (closed) \textbf{interval} in~$T$. (Open and half-open intervals in~$T$ are defined analogously.) If $t < t'$ but there is no point between $t$ and~$t'$, we call $t'$ a \textbf{successor} of~$t$ and $t$ the \textbf{predecessor} of~$t'$; if $t$ is not a successor of any point it is called a \textbf{limit}.

Let $T$ be an order tree. A graph $G$ is a \textbf{$T$-graph} if the endvertices of all edges of $G$ are comparable in the tree order of $T$, and the set of lower neighbours of any point $t$ is cofinal in $\mathring{\lceil t \rceil}$. Note that all $T$-graphs are connected. For more information on $T$-graphs, see \cite{NormalTreeOrdersForInfiniteGraphs, pitz2024applications}.

\subsection{Normal partition trees}

We now recall the concept of a \textbf{normal partition tree} due to Brochet and Diestel \cite{NormalTreeOrdersForInfiniteGraphs}, a powerful tool to structure infinite graphs, see e.g.\ \cite{kurkofka2021representation,Pitz}.

Let $G$ be a graph, $T$ an order tree, and $\cV =  (\,V_t\colon t\in T\,)$ a family of non-empty \textbf{bags} $V_t \subseteq V(G)$ indexed by the nodes $t$ of $T$. The pair $(T, \cV)$ is called a  \textbf{normal partition tree} of~$G$ if

\begin{enumerate}[label=(\alph*)]
    \item\label{item_NPT0} $\cV$ forms a partition of $V(G)$,
    \item\label{item_NPT2} each \textbf{part} $G[V_t]$ is connected, and 
    \item\label{item_NPT1} the contraction minor $\dot{G}:=G/\cV$ (where we contract each $V_t \in \cV$ to a single vertex, with all arising parallel edges and loops deleted) is a $T$-graph
\end{enumerate}

Having a fixed normal partition tree $ (T,\cV)$ in mind, for a given vertex $v\in G$ we write $t(v)$ for the node $t\in T$ such that $v\in V_t$, which is well-defined by \ref{item_NPT0}.

A normal partition tree approximates $G$ best if the partition sets $V_t$ are small. To capture this intuition, we say that a normal partition tree $ (T,\cV)$ is \textbf{narrow} if $|V_t| \leq  \cf (\height (t))$, and it is \textbf{slim} if $|V_t|\leq |\height(t)| + \aleph_0$ for every $t\in T$.

A foundational result due to Brochet \& Diestel, which we will generalise below, says that every connected graph has a narrow normal partition tree, see \cite[Theorem~4.2]{NormalTreeOrdersForInfiniteGraphs}.

\subsection{Normal semi-partition trees}

Just as one can generalise the concept of a normal spanning tree to subtrees that are not necessarily spanning, we now generalise the concept of a normal partition tree to semi-partitions where the bags $\cV$ are still disjoint, but their union does not necessarily cover all of $V(G)$.  As before, let $G$ be a graph, $T$ an order tree, and $\cV =  (\,V_t\colon t\in T\,)$ a family of non-empty \textbf{bags} $V_t \subseteq V(G)$ indexed by the nodes $t$ of $T$. The pair $(T, \cV)$ is called a \textbf{normal semi-partition tree} of~$G$ if

\begin{enumerate}[label=(\alph*)]
    \item\label{item_SNPT0} the bags in $\cV$ are pairwise disjoint,
    \item\label{item_SNPT2} each \textbf{part} $G[V_t]$ is connected, 
    \item\label{item_SNPT1}\label{GPunkt} the contraction minor $\dot{G}:=G[\bigcup \cV]/\cV$ (where we delete all vertices outside of $\bigcup \cV$ and then contract each $V_t \in \cV$ to a single vertex) is a $T$-graph, and
    \item\label{item_SNPT3}\label{comparable} for every $\bigcup \cV$-path $P$ in $G$ with endvertices $u \in V_{t}$ and $v \in V_{t'}$, the nodes $t$ and $t'$ are comparable in the tree order of $T$.
\end{enumerate}

 A normal semi-partition tree $ (T,\cV)$ is \textbf{narrow} if $|V_t| \leq \cf (\height (t))$ for every $t\in T$, and it is \textbf{slim} if $|V_t|\leq \height(t) + \aleph_0$ for every $t\in T$.

Let $U \subseteq V(G)$ be a set of vertices in $G$. If for every $t \in T$ we have $V_t \cap U \neq \emptyset$, we say that the semi-partition tree $(T,\cV)$ is \textbf{$U$-rooted}.

Having a fixed normal semi-partition tree $ (T,\cV)$ in mind, for a given vertex $v\in \bigcup \cV$ we write $t(v)$ for the node $t\in T$ such that $v\in V_t$, which is well-defined by \ref{item_SNPT0}. For a set of nodes $S \subset T$, we write $V(S) = \bigcup_{t \in S} V_t$ and $G(S) = G[V(S)]$ for the subgraph of $G$ induced by $V(S)$.

We now record some elementary properties of normal semi-partition trees, which generalise the corresponding properties for normal partition trees of \cite[\S2]{NormalTreeOrdersForInfiniteGraphs}.

\begin{lemma}
\label{lem_Tgraphproperties}
    Let $G$ be a connected graph and $(T,\cV)$ be a normal semi-partition tree in $G$. Then:
    
    \begin{enumerate}
        \item For incomparable points $t, t' \in T$ the set $G(\lceil t \rceil \cap \lceil t' \rceil)$ separates $V_t$ from $V_{t'}$ in $G$.
        \item\label{itemT2} Every connected subgraph $H$ of $G$ that intersects with $G(T)$ there exists unique $T$-minimal element $t$ such that $V_t \cap H \neq \emptyset$.
        \item\label{itemT3} If $T' \subseteq T$ is down-closed, then every component of $G(T) - G(T')$ is spanned by a set $G(\lfloor t \rfloor)$ for some $t$ minimal in $T - T'$.
        \item\label{itemT4} If $T' \subseteq T$ is down-closed, then every component $C$ of $G - G(T')$ that meets $G(T)$ restricts to a component $C \cap G(T)$ of $G(T) - G(T')$.
    \end{enumerate}
    
\end{lemma}

\begin{proof}
    Let $G$ be a connected graph and $(T,\cV)$ be a normal semi-partition tree in $G$.
    
    For property $(1)$, let $t, t' \in T$ be two incomparable points. Then, $t, t' \notin \lceil t \rceil \cap \lceil t' \rceil$. Let $P$ be a $V_t$--$V_{t'}$ path in $G$. We show that $P$ meets $G(\lceil t \rceil \cap \lceil t' \rceil)$. Let $\tau := \{t \in T: V_t \cap P \neq \emptyset\}$. Let $t = t_1, \dots, t_n = t'$ be a minimal sequence of points in $\tau$ such that $t_i$ and $t_{i + 1}$ are comparable in the tree order for all $i < n$. Such a sequence exists since $\tau$ itself with the natural order obtained by \textit{following} $P$ is already such a sequence by property \ref{item_SNPT3}. We claim that the sequence has the following form:
    $$
    t = t_1 > \dots > t_k < \dots < t_n = t'.
    $$
    Indeed, if there is an $i$ such that $t_{i - 1} < t_i > t_{i + 1}$, then $t_{i - 1}, t_{i+1} \in \downcl{t_i}$ are comparable, so deleting $t_i$ yields a shorter sequence. Hence, $t_k \in \lceil t \rceil \cap \lceil t' \rceil$ and thus $P$ meets $G(\lceil t \rceil \cap \lceil t' \rceil)$.
    
    For property $(2)$, let $H$ be a connected subgraph of $G$ that intersects with $G(T)$. Suppose for a contradiction that there are two minimal elements $t$ and $t'$ such that $V_t \cap H \neq \emptyset \neq V_{t'} \cap H$. Since $H$ is connected, there is a $G(\lfloor t \rfloor)$--$G(\lfloor t' \rfloor)$ path $P$ in $H$. By (1), the path $P$ meets $G(\lceil t \rceil \cap \lceil t' \rceil)$, a contradiction to the minimality of $t$ and $t'$.

    For property $(3)$, let  $T' \subseteq T$ be down-closed and $C$ a component of $G(T) - G(T')$. We show that $C = G(\lfloor t \rfloor)$ for some point $t$ that is minimal in $T - T'$. Since $C \cap G(T) \neq \emptyset$, there exists by assertion $(2)$ a unique $T$-minimal element $t \in T$ such that $V_t \cap C \neq \emptyset$. In particular $t \notin T'$, as otherwise $C \cap G(T') \supseteq C \cap V_t \neq \emptyset$. Since the inclusion $C \subseteq G(\lfloor t \rfloor)$ is obvious by the uniqueness of $t$, it remains to show that $C \supseteq G(\lfloor t \rfloor)$. But this follows from the fact that  $G(\lfloor t \rfloor)$ is connected, disjoint from $G(T')$, and meets $C$; so is included in the component $C$ by maximality of $C$.

    For property $(4)$, consider a component $C$ of $G - G(T')$ that meets $G(T)$. By (2), there is a unique minimal element $t \in T$ such that $V_t \cap C \neq \emptyset$. By uniqueness of $t$, it follows $C \cap G(T) \subseteq G(\upcl{t})$. But conversely, $G(\upcl{t})$ spans a connected subset of $G(T) - G(T')$ by (3), so $G(\upcl{t}) \subseteq C$, which implies $C \cap G(T) = G(\upcl{t})$ as desired.
\end{proof}

\subsection{Existence of normal semi-partition trees}

We already know from \cite{NormalTreeOrdersForInfiniteGraphs} that every connected graph admits a narrow normal partition tree. By relaxing `narrow' to `slim', we can ensure that $G(T)$ and, in fact, $G(T')$ for all rooted subtrees $T'$ of $T$ have well-connected adhesion sets.

\begin{thm}
\label{3.4}
    Let $U \subseteq V(G)$ be a prescribed set of vertices in a connected graph $G$. Then there is a slim normal semi-partition tree $(T,\cV)$ of $G$ such that
    
    \begin{enumerate}
        \itemlabel{3.4.0}{$(1)$}$U \subset G(T)$,
        \itemlabel{3.4.1}{$(2)$} $T$ is $U$-rooted, 
        \itemlabel{3.4.2}{$(3)$} for each rooted subtree $T' \subset T$, the subgraph $G(T')$ has well-connected adhesion sets in $G$.
    \end{enumerate}
    
\end{thm}

\begin{proof}
    Fix an enumeration $\set{u_i}{i < \kappa}$ of $U$.
    We construct the desired normal semi-partition tree $(T,\cV)$ recursively as an increasing union of down-closed subtrees $\set{T_i}{i \leq \kappa}$ of $T$ such that each $T_i$ satisfies $\{u_j \colon j < i \} \subseteq G(T_i)$ and properties \ref{3.4.1} and \ref{3.4.2}. In the end, $T = T_\kappa$ is as desired.
    
    We begin by setting $T_0 = \emptyset$. In the successor step $i \mapsto i+1$, suppose that $T_i \subset T$ has already been defined. If $u_i \in G(T_i)$, then put $T_{i+1} = T_i$. Otherwise, there is a unique component $D$ of $G - G(T_i)$ with $u_i \in V(D)$. Consider
    $$
    C := \lceil \set{t \in T_i}{ N(D) \cap V_t \neq \emptyset } \rceil \subset T_i.
    $$
    Since $T_i$ is normal, property \ref{comparable} implies that $C$ is a chain in $T_i$. For each vertex in $N(D) \subseteq G(C)$ pick one neighbour inside $D$, and let $Z_D$ be the set of all these neighbours. Let $W_D := V(C) \cup Z_D \cup \{u_i\}$. Then Lemma~\ref{lem_wellconnected_new} applied inside the graph $G_D = G[V(C) \cup D]$ with input $W_D \supseteq V(C)$  yields a set of vertices $\hat{W}_D \supseteq W_D$. We obtain $T_{i + 1}$ from $T_i$ by placing a fresh node $t$ immediately above $C$ (and incomparable to all other nodes of $T_i)$ with corresponding part $V_t := \hat{W}_D \setminus V(C)$.

    We verify that $T_{i + 1}$ is a normal semi-partition tree: By construction, property \ref{item_SNPT0} holds. Property \ref{GPunkt} follows from choice of $Z_D$ and property \ref{comparable} is straightforward. Since $V_t$ induces a connected subgraph by property (3) of Lemma~\ref{lem_wellconnected_new}, we get property \ref{item_SNPT2} for $T_{i+1}$. By construction, $T_{i + 1}$ is $U$-rooted, giving property \ref{3.4.1}, and by construction we have $\{u_j \colon j < i + 1\} \subseteq  G(T_{i + 1})$. 
    
    For property \ref{3.4.2}, let $T'$ be a rooted subtree of $T_{i+1}$, and consider an arbitrary $G(T')$-path $P$ in $G$. If both endvertices of $P$ belong to $G(T_i)$, then we are done by the induction assumption. Hence, we may assume that one of the end-points of $P$ belongs to $V_t$, and hence $t \in T'$. Since $G(T')$ is induced, $P$ has inner vertices, which must be contained in the component $D$ of $G - G(T_i)$. But then $P$ is a $\hat{W}_D$-path in $G_D$, and the assertion follows by choice of $\hat{W}_D \subset G(T')$, that is property (2) of Lemma~\ref{lem_wellconnected_new}.

    Finally, to see that $T_{i + 1}$ is slim, we must check that $|V_t| \leq |\height(t)| + \aleph_0$. First observe that  $\mathring{\lceil t \rceil} = C$ implies $|\height(t)| = |C|$. Moreover, we have $$|Z_D| \leq |N(D)| \leq |V(C)| \leq |C| + \aleph_0$$ where the last inequality follows since $T_i$ is slim by induction assumption. It follows
    $$
    |V_{t}| \leq |\hat{W}_D| \leq |W_D| + \aleph_0 = |V(C) \cup Z_D \cup \{u_i\}| + \aleph_0 \leq  |C| + \aleph_0 =  |\height(t)| + \aleph_0
    $$
    where the second inequality holds by property (1) of Lemma~\ref{lem_wellconnected_new}, and the last inequality follows from the previous equation.
    
    For limits $\l \leq \kappa$ define
    $$
    T_\l := \bigcup_{i < \l} T_i.
    $$
    By routine arguments, it follows that $T_\l$ is a slim $U$-rooted normal semi-partition tree of $G$ such that for all $j < \l$ every $u_j \in G(T_\l)$ satisfying \ref{3.4.1}. Now, we check property \ref{3.4.2}. Let $T'$ be an arbitrary rooted subtree of $T_{\l}$, and consider an arbitrary $G(T')$-path $P$ in $G$ with end-vertices say $v$ and $w$. Then there is $i < \ell$ with $v,w \in G(T_i)$, so $P$ is already a $G(T' \cap T_i)$ path. By the induction assumption, $v$ and $w$ have infinite connectivity inside $G(T' \cap T_i) \subset G(T')$ as desired.
\end{proof}

\section{The decomposition lemma}
\label{sec_declemma}

Let $H$ be an induced subgraph of a graph $G$, and let $U \subseteq V(G)$ be a prescribed set of vertices. Any component of $G-H$ that contains a vertex from $U$ is called a \textbf{$U$-component} of $G-H$. We say that $H$ has \textbf{finite adhesion in $G$ towards $U$} if every $U$-component of $G-H$ has only finite neighbourhood in $H$.

In this section, our main target is to prove the following decomposition lemma. 

\begin{lemma}[Decomposition lemma]
\label{3.6}
    Let $G$ be a connected graph. Let $U \subseteq V(G)$ be an uncountable set of vertices and suppose that every $U$-rooted minor of $G$ with countable branch sets has countable colouring number.
    
    Then any slim $U$-rooted normal semi-partition tree $(T,\cV)$ in $G$ with $U \subseteq G(T)$ can be written as a continuous, increasing union
    $$
    T = \bigcup_{i < \cf(\kappa)}T_i
    $$
    of infinite, $<|T|$-sized rooted subtrees $T_i$ such that all graphs $G(T_i)$ have finite adhesion in $G$ towards $U$.
\end{lemma}

Given the decomposition lemma together with our notion of well-connected adhesion sets, the proof of our main Theorem~\ref{thm_main2} will be surprisingly simple;w see Section~\ref{sec_finalproof} below. The upcoming subsections contain all ingredients for the proof of the decomposition lemma.

\subsection{Concrete obstructions}

In our proof, we will use that the following three types of graphs have uncountable colouring number and therefore cannot appear as $U$-rooted minors of $G$:

\begin{enumerate}[label=(\roman*)]
	\item\label{barricades} A \textbf{barricade}, i.e.\ a bipartite graph with bipartition $(A,B)$ such that $|A| < |B|$ and every vertex of $B$ has infinitely many neighbours in $A$, cf.\ \cite[Lemma~2.4]{TheColouringNumberOfInfiniteGraphs}.
    \item\label{omega1-graphs} An \textbf{$\omega_1$-graph}, i.e.\ a $T$-graph for $T = \omega_1$, the first uncountable ordinal cf.\ \cite[Proposition~3.5]{NSTandAronszajntrees}.
	\item\label{AT-graphs} An \textbf{Aronszajn tree-graph}, i.e.\ a $T$-graph for an Aronszajn tree $T$, cf.\ \cite[Theorem~7.1]{NSTandAronszajntrees}.
\end{enumerate}

\begin{lemma}
\label{lem_noOmega_1chains}
    Let $U \subseteq V(G)$ be a set of vertices in a connected graph $G$, and suppose that every $U$-rooted minor of $G$ with countable branch sets has countable colouring number.
    
    Then all branches of a slim $U$-rooted normal semi-partition tree $(T,\cV)$ in $G$ are at most countable; in particular, all branch sets $V_t$ in $G$ are countable. 
\end{lemma}

\begin{proof}
    If $T$ contains an uncountable branch, then an initial segment of order type $\omega_1$ would form an $\omega_1$-graph minor in $G$ that is $U$-rooted and has countable branch sets (since $T$ is slim). This contradicts item \ref{omega1-graphs} above. 
    
    In particular, $\height(t) < \omega_1$ for all $t \in T$, and so the second assertion follows from the property that our normal semi-partition tree is slim.
\end{proof}

\begin{lemma}\label{barricade}
    Let $G$ be a graph. Let $U \subseteq V(G)$ be a set of vertices of $G$. Let $G$ have a minor with countable branch sets that is a barricade with bipartition $(A, B)$ such that the $B$-side is $U$-rooted, i.e. for every vertex $b \in B$ the corresponding branch set in $G$ contains a vertex of $U$. Then there is a $U$-rooted barricade of $G$ with countable branch sets.
\end{lemma}

\begin{proof}
    Let $G$ be a graph. Let $U \subseteq V(G)$ be a set of vertices of $G$. Let $H$ be a barricade minor of $G$ with countable branch sets with bipartition $(A, B)$ such that the $B$-side is $U$-rooted. With \cite[Lemma 2.4]{TheColouringNumberOfInfiniteGraphs} we may assume that every vertex $a \in A$ has more than $\cardinality{A}$ many neighbours in $B$. From this degree condition, it readily follows that there is a matching $M$ of $A$ in $H$. By contracting all edges of $M$ in $H$, we obtain a barricade minor $H'$ of $H$ with branch sets of size $2$ such that every branch set of $H'$ contains a vertex of $B$; so the minor $H'$ of $G$ has countable branch sets and is now $U$-rooted as desired.
\end{proof}

\subsection{The closure lemma}

\begin{lemma}[Closure Lemma]
\label{3.7}
    Let $U \subseteq V(G)$ be a set of vertices in a connected graph $G$. Suppose that every $U$-rooted minor of $G$ with countable branch sets has countable colouring number. Let $(T,\cV)$ be a slim $U$-rooted normal semi-partition tree in $G$ with $U \subseteq G(T)$. Then every infinite set $X \subset T$ is included in a rooted subtree $T' \subseteq T$ of size $\cardinality{X} = \cardinality{T'}$ such that $G(T')$ has finite adhesion in $G$ towards $U$.
\end{lemma}

This lemma is a local analogue of \cite[Lemma 3.7]{Pitz}. Since we must construct barricades and Aronszajn trees that are $U$-rooted -- which differs from the proof in \cite[Lemma 3.7]{Pitz} --  we will give the entire proof here for convenience of the reader.

\begin{proof}
    For a connected subgraph $D \subset G$ with $D \cap U \neq \emptyset$ write $t_D$ for the by Lemma~\ref{lem_Tgraphproperties}(\ref{itemT2}) unique $T$-minimal element such that $V_{t_D}$ meets $D$. We recursively build a $\subseteq$-increasing sequence $\{T_i \colon i < \omega_1\}$ of rooted subtrees of $T$ by letting $T_0 = \lceil X \rceil_T$, defining $$ T_{i+1} = T_i \cup \{t_D \colon D \text{ a $U$-component of } G - G(T_i) \text{ with } |N(D) \cap G(T_i)| = \infty\}$$ at successor steps, and $T_\ell = \bigcup_{i < \ell} T_i$ for limit ordinals $\ell < \omega_1$. Finally we set $T' = \bigcup_{i < \omega_1} T_i$. Clearly, $T'$ is a rooted subtree of $T$ including $X$. 
    
    To see that $G(T')$ has finite adhesion in $G$ towards $U$, suppose for a contradiction that there is a $U$-component $D$ of $G-G(T')$ with $| N(D) \cap G(T') | = \infty$. Then there is some $i_0 < \omega_1$ such that $|N(D) \cap G(T_{i_0})| = \infty$. Hence for all $i_0 \leq i < \omega_1$, the unique component $D_i$ of $G - G(T_i)$ containing $D$ is also a $U$-component that satisfies $|N(D_i) \cap G(T_i)| = \infty$. Then $\{t_{D_i} \colon i_0 \leq i < \omega_1\}$ forms an uncountable chain in $T$, contradicting Lemma~\ref{lem_noOmega_1chains}.
    
    To see that $|T'| = |X|$, observe that since $T$ contains no uncountable chains by Lemma~\ref{lem_noOmega_1chains}, we have $|T_0| = |X|$. We now prove by transfinite induction on $i< \omega_1$ that $|T_i| = |X|$. The cases where $i$ is a limit are clear, so suppose $i = j+ 1$. By the induction hypothesis, $|T_j| = |X|$. Suppose for a contradiction that $|T_{j+1}|>|T_j|$. We construct a barricade minor of $G$ with countable branch sets such that the $B$-side is $U$-rooted as follows: Define
    $$
    A := V(T_j).
    $$
    By the in-particular part in Lemma~\ref{lem_noOmega_1chains} we have $|A| = |T_j|$. For $B$, consider all $U$-components $D$ of $G - G(T_j)$ with $t_D \in T_i - T_j$. By definition of $t_D$ it is true that $\cardinality{N(D)} = \infty$. Let $N \subseteq N(D)$ be a countable subset of $N(D)$. Find for every $n \in N$ a neighbour $d_n \in N(n) \cap D$. Also, let $u_D \in U \cap D$. Then, $u_D$ and all $d_n$ are at most countable many vertices in $D$. Find a tree $T_D$ in $D$ of countable size that contains these vertices. Moreover, $|T_i|  > |T_j| = |A|$ implies that $|B| > |A|$, so we have found a barricade minor with countable branch sets whose $B$-side is $U$-rooted. Using \cref{barricade}, we also find a barricade as a minor of $G$ with countable branch sets such that both sides are $U$-rooted, contradicting \ref{barricades} above.

    If $X$ is uncountable, then $|T'| = | \bigcup_{i < \omega_1} T_i | = \aleph_1 \cdot |X| = |X|$. So suppose for a contradiction that $X$ is countable and $|T'| = \aleph_1$. Contracting the countable rooted subtree $T_0$ to a vertex $r$ in $T'$ gives rise to an order tree $T''$ with root $r$. Since $T_0 \subset T'$ is a rooted subtree and so $G(T_0)$ is connected, this contraction results in a minor $G''$ of $G(T)$ with countable branch sets that is a $T''$ graph. By construction, nodes in $T_{i} \setminus \bigcup_{j < i} T_j$ for $i \geq 1$ belong to the $i$th level of $T''$, and hence all levels of $T''$ are countable. Finally, since $T''$ like $T'$ and $T$ contains no uncountable chains, it follows that $T''$ is an Aronszajn tree. Since $G'' \preceq G(T)$, we have found an Aronszajn tree $U$-rooted minor of $G$ with countable branch sets, a contradiction.
\end{proof}

\subsection{Proof of the decomposition lemma}

The proof of this decomposition lemma is now identical to the corresponding proof in \cite{Pitz} once we replace \cite[Lemma~3.7]{Pitz} by our new Closure Lemma~\ref{3.7}. For convenience of the reader, we reprint the full argument in the appendix.

\section{The hard implication}
\label{sec_finalproof}

\begin{thm}
    Let $G$ be a connected graph. Let $U \subseteq V(G)$ be a set of vertices of $G$. Then, $U$ is normally spanned in $G$ if and only if all $U$-rooted minors in $G$ with countable branch sets have countable colouring number.
\end{thm}

\begin{proof}
    In remains to prove the backwards direction. Let $U \subseteq V(G)$ be a set of vertices in a connected graph $G$, and suppose that every $U$-rooted minor of $G$ with countable branch sets has countable colouring number. Fix a slim normal semi-partition tree $(T,\cV)$ in $G$ as in Theorem~\ref{3.4}. We prove by induction on $\kappa=|T|$ that the subgraph $G(T)$ of $G$ has a normal spanning tree. Then we are done, since by Lemma~\ref{lem_normalinsubgraphs}, this tree is also normal in $G$ (and covers $U$ by property~\ref{3.4.0} Theorem~\ref{3.4}).
    
    If $\kappa$ is countable, then $G(T)$ is countable (since $T$ is slim), and hence has a normal spanning tree.
    Now suppose that $\kappa$ is uncountable, and the theorem holds for all smaller cardinals. By the Decomposition Lemma~\ref{3.6}, we get a continuous increasing chain $\set{T_i}{i < \sigma}$ of infinite, $< \kappa$-sized rooted subtrees $T_i$ of $T$ such that each $G(T_i)$ has well-connected adhesion sets (by property~\ref{3.4.2} in Theorem~\ref{3.4}) and has finite adhesion in $G$ towards $U$.
    
    We construct by recursion on $i < \sigma$ a sequence of normal trees $\set{S_i}{i < \sigma}$ in $G$ all with the same root and all extending each other, so that each $S_i$ is a normal spanning tree of $G_i :=G(T_i)$.
     In the end, define
    $$
    S := \bigcup_{i < \sigma} S_i.
    $$
    Then clearly, $S$ is the desired normal spanning tree of $G(T)$.
    
    It remains to describe the recursive construction. First, let $i = 0$. Since $\cardinality{T_0} < \kappa$, we find a normal spanning tree $S_0$ of $G_0$. Since $G_0$ has well-connected adhesion sets, it follows from Lemma~\ref{lem_normalinsubgraphs} that $S_0$ is normal in $G$.

    Next, suppose that $i = \l$ is a limit. Define
    $
    S_\l := \bigcup_{j < \l}S_j.
    $
    Then $S_\l$ is a normal tree in $G$ extending all $S_j$ for $j < \l$. Also, $V(S_\l) = V(G_\l)$, so $S_\l$ is a normal spanning tree of $G(T_\l)$.

    For the successor step, suppose that for some $i < \sigma$ we have already defined a normal spanning tree $S_i$ of $G_i$ that is normal in $G$. In order to construct $S_{i+1}$, consider a component $D$ of $G_{i+1} - G_i$. Since $S_i$ is normal in $G$,  it follows that $N(D)$ lies on a chain of $S_i$. Since $G_i$ has finite adhesion towards $U$, and $D$ meets $U$ (as $T$ is $U$-rooted), this chain $N(D)$ is finite. Thus there exists a maximal element $s_{D} \in N(D)$ in the tree order of $S_i$. Choose a neighbour $r_{D}$ of $s_{D}$ in $D$, and write $e_{D}$ for the edge $s_{D}r_D$.

    By induction assumption, $G_{i+1}$ has a normal spanning tree. So, by Theorem~\ref{thm_Jung} there is a normal spanning tree $S_{D}$ of $D$ with prescribed root $r_{D}$. Then
    $$
    S_{i+1} := S_i \cup \bigcup_{D} S_{D} \cup \bigcup_{D} \singleset{e_{D}}
    $$
    is a normal spanning tree of $G_{i+1}$. The construction is complete.
\end{proof}

\bibliographystyle{plainurl}
\bibliography{Ref}

\appendix

\section{Proof of the decomposition lemma}

For completeness, we provide the proof of the Decomposition Lemma~\ref{3.6}. The proof distinguishes two cases depending on whether $\kappa = |T|$ is regular or singular.

\begin{proof}[Proof of Lemma~\ref{3.6} for regular uncountable $\kappa$]
    Since $\dot{G}$ is a $U$-rooted minor of $G$ with countable branch sets (Lemma~\ref{lem_noOmega_1chains}), it has countable colouring number. Fix a well-order $\dot{\leq}$ of $V(\dot{G})$ so that every vertex has only finitely many neighbours preceding it in $\dot{\leq}$. We may choose $\dot{\leq}$ to be of order type $|\dot{G}|$, see e.g.\ \cite[Corollary~2.1]{EGJKP19}. 
    
    Fix an enumeration $V(T)=\{t_i \colon i < \kappa\}$. We recursively define a continuous increasing sequence $\{T_i \colon i < \kappa\}$ of rooted subtrees of $T$ with 
    
    \begin{enumerate}
    	\item $t_i \in T_{i+1}$ for all $i < \kappa$,
    	\item each $G(T_i)$ has finite adhesion in $G$ towards $U$,
    	\item the vertices of $T_i$ form a proper initial segment of $(V(\dot{G}),\dot{\leq})$
    \end{enumerate}
    
    Let $T_0 = \emptyset$. In the successor step, suppose that $T_i$ is already defined. Let $T_i^0 := T_i \cup \lceil t_i \rceil$. At odd steps, use the Closure Lemma~\ref{3.7} to fix  a rooted subtree $T_i^{2n+1}$ of $T$ including $T_i^{2n}$ of the same size as $T_i^{2n}$ so that $G(T_i^{2n+1})$ has finite adhesion in $G$ towards $U$. At even steps, let $T_i^{2n+2}$ be the smallest subtree of $T$ including the down-closure of $T_i^{2n+1}$ in $(V(\dot{G}),\dot{\leq})$.Define $T_{i+1} = \bigcup_{n \in \NN}T_i^n$. By construction, $T_{i+1}$ is a rooted subtree of $T$ with $t_i \in T_{i+1}$, and $T_{i+1}$ forms an initial segment of $(V(\dot{G}),\dot{\leq})$. To see that this initial segment is proper, one inductively verifies that $|T_i^{n}| < \kappa$; since $\kappa$ has uncountable cofinality, this also gives $|T_{i+1}| < \kappa$. 

    Hence, it remains to show that $G(T_{i+1})$ has finite adhesion in $G$ towards $U$. Suppose otherwise that there exists a $U$-component $D$ of $G-G(T_{i+1})$ with infinitely many neighbours in $G(T_{i+1})$. If we let $d =t_D$, then $t(N(D)) \subset \mathring{\lceil d \rceil}_T$ holds by definition of a normal semi-partition tree. We claim that $d$ must be a limit of $T$. Indeed, for any $x <_T d$, Lemma~\ref{lem_Tgraphproperties}(\ref{itemT3}) implies that $x \in T_i^{2n+1}$ for some $n \in \NN$. Since $G(T_i^{2n+1})$ has finite adhesion towards $U$, it follows that $N(D) \cap G(T_i^{2n+1})$ is finite. In particular, only finitely many neighbours $v \in N(D)$ satisfy $t(v) \leq_T x$. Hence, at least one neighbour $v \in N(D)$ satisfies $x <_T t(v) <_T d$; so $d$ is a limit.

    By the definition of a $T$-graph, $d$ has infinitely many  $\dot{G}$-neighbours below it, and hence in $T_{i+1}$. However, since $T_{i+1}$ forms an initial segment in $(V(\dot{G}),\dot{\leq})$ not containing $d$, it follows that $d$ is preceded by infinitely many of its neighbours in $\dot{\leq}$, contradicting the choice of $\dot{\leq}$.

    For limits $\ell < \kappa$ we define $T_\ell = \bigcup_{i < \ell} T_i$. As above, $T_\ell$ is a rooted subtree of $T$ that forms a proper initial segment in $(V(\dot{G}),\dot{\leq})$ such that $G(T_\ell)$ has finite adhesion towards $U$.
\end{proof}

\begin{proof}[Proof of Lemma~\ref{3.6} for singular uncountable $\kappa$]
    Let us enumerate $V(T) = \{t_i \colon i < \kappa\}$ and fix a continuous increasing sequence $\{\kappa_i \colon i < cf(\kappa)\}$ of cardinals with limit $\kappa$, where $\kappa_0 > cf(\kappa)$ is uncountable.
    We build a family $$\{T_{i,j} \colon i < cf(\kappa), \; j < \omega_1\}$$ of rooted subtrees of $T$ with $G(T_{i,j})$ of finite adhesion in $G$ towards $U$, with each $T_{i,j}$ of size $\kappa_i$. This will be done by a nested recursion on $i$ and $j$. When we come to choose $T_{i,j}$, we will already have chosen all $T_{i',j'}$ with $j' < j$, or with both $j' = j$ and $i' < i$. Whenever we have just selected such a subtree $T_{i,j}$, we immediately fix an arbitrary enumeration $\set{t^k_{i,j}}:{k < \kappa_i}$ of this tree. We impose the following conditions on this construction:
    
    \begin{enumerate}
    	\item $\{t_k \colon k < \kappa_i\} \subset T_{i,0}$ for all $i$,
    	\item $\bigcup \{T_{i',j'} \colon i' \leq i, j' \leq j\} \subset T_{i,j}$ for all $i$ and $j$,
    	\item $\bigcup \{t^k_{i',j} \colon k < \kappa_i\} \subset T_{i,j+1}$ for all $i < i' < cf(\kappa)$ and $j$.
    \end{enumerate}
    
    These three conditions specify some collection of $\kappa_i$-many vertices which must appear in $T_{i,j}$. By Lemma~\ref{3.7} we can extend this collection to a subtree $T_{i,j}$ of the same size such that $G(T_{i,j})$ has finite adhesion in $G$ towards $U$. This completes the description of our recursive construction.

    Condition (3) ensures that
    
    \begin{enumerate}
    	\item[(4)] $T_{\ell,j} \subset \bigcup_{i < \ell} T_{i,j+1}$ for all limits $\ell < cf(\kappa)$ and all $j$.
    \end{enumerate}
    
    In fact, since $\kappa_\ell = \bigcup_{i < \ell} \kappa_i $ by the continuity of our cardinal sequence, it follows that $T_{\ell,j} = \{t^k_{\ell,j} \colon k < \kappa_\ell\} = \bigcup_{i < \ell} \{t^k_{\ell,j} \colon k < \kappa_i\} \subset \bigcup_{i < \ell} T_{i,j+1}$.
    Now for $i < cf(\kappa)$, the set $T_i = \bigcup_{j < \omega_1} T_{i,j}$ yields a subgraph $G(T_i)$ of finite adhesion in $G$ towards $U$ (as an increasing $\omega_1$-union of subgraphs of finite adhesion towards $U$ has itself finite adhesion towards $U$). In addition, the sequence $\{T_i \colon i < cf(\kappa)\}$ increases by (2) and is continuous by (4). 
\end{proof}

\end{document}